\pdfoutput=1

\documentclass[11pt]{amsart}

\usepackage[T1]{fontenc}
\usepackage[utf8]{inputenc}
\usepackage{lmodern}

\usepackage{amsmath,amssymb,amsfonts,amsthm}
\usepackage{mathtools}
\usepackage{enumitem}
\usepackage{microtype}
\usepackage{titlesec}
\usepackage{etoolbox}
\usepackage[colorlinks=true,
            linkcolor=blue,
            citecolor=blue,
            urlcolor=blue]{hyperref}

\usepackage[margin=1.15in]{geometry}

\titleformat{\section}[hang]
  {\normalfont\large\bfseries}
  {\thesection.}{0.65em}{}
\titlespacing*{\section}
  {0pt}{3.2ex plus 0.8ex minus 0.2ex}{1.2ex plus 0.2ex}

\titleformat{\subsection}[hang]
  {\normalfont\normalsize\bfseries}
  {\thesubsection.}{0.6em}{}
\titlespacing*{\subsection}
  {0pt}{2.2ex plus 0.6ex minus 0.2ex}{0.8ex plus 0.2ex}

\makeatletter
\def\@settitle{%
  \begin{center}%
    \baselineskip18\p@\relax
    \Large\bfseries
    \@title
  \end{center}%
}
\AtBeginDocument{\renewcommand{\uppercasenonmath}[1]{}}
\patchcmd{\@setauthors}
  {\MakeUppercase{\authors}}
  {\authors}
  {}{}
\patchcmd{\maketitle}
  {\@nx\MakeUppercase}
  {\@nx\@firstofone}
  {}{}

\makeatother

\theoremstyle{plain}
\newtheorem{theorem}{Theorem}[section]
\newtheorem{proposition}[theorem]{Proposition}
\newtheorem{lemma}[theorem]{Lemma}

\theoremstyle{definition}

\newtheorem{example}[theorem]{Example}

\theoremstyle{remark}
\newtheorem{remark}[theorem]{Remark}

\newcommand{\R}{\mathbb{R}}
\newcommand{\dd}{\,d}

\newcommand{\Sch}[2]{\left\{#1,#2\right\}}
\newcommand{\kaff}{\kappa_{\mathrm{aff}}}
\newcommand{\Jac}{J_{\Psi}}
\newcommand{\env}{\gamma}
\newcommand{\Ob}{S_h}
\newcommand{\Res}{R_h}
\newcommand{\Web}{\mathcal W}

\newcommand{\doi}[1]{\url{https://doi.org/#1}}

\title[Schwarzian Residue in Tangent Lagrangian \(2\)-Webs]
{Schwarzian Residue of the Samuelson Obstruction in Tangent Lagrangian
\mbox{\(2\)-Webs}}

\author[\NoCaseChange{Yasuhiro Kurokawa}]{Yasuhiro Kurokawa}

\address{\textnormal{Department of Architecture, School of Architecture,
Shibaura Institute of Technology, 3-7-5 Toyosu, Koto-ku,
Tokyo 135-8548, Japan.}}

\email{kurokawa@sic.shibaura-it.ac.jp}

\date{\today}



\hypersetup{
  pdftitle={Schwarzian Residue of the Samuelson Obstruction in Tangent Lagrangian 2-Webs},
  pdfauthor={Yasuhiro Kurokawa}}

\begin{document}

\begin{abstract}
The Samuelson condition, a classical area-ratio condition for planar
Lagrangian \(2\)-webs, is equivalent in local web coordinates \((u,v)\) to
\(\partial_u\partial_v\log|J_F(u,v)|=0\),
where \(F\) is the inverse web-coordinate map and \(J_F\) is its Jacobian.
For the tangent-line family
\(L_t:\, y=tx+h(t)\), with \(h\) smooth on an open interval,
let \(\Psi(u,v)\), for \(u\neq v\), be the intersection map of \(L_u\)
and \(L_v\), write \(\Jac\) for its Jacobian, and set
\(\Ob(u,v) := \partial_u\partial_v\log|\Jac(u,v)|\)
where \(\Jac(u,v)\neq0\).
Near each diagonal point \((t_0,t_0)\) with \(h''(t_0)\neq0\), the
obstruction admits the decomposition
\[
        \Ob(u,v)
        =
        \frac{1}{(v-u)^2}
        +
        \Res(u,v),
\]
where \(\Res\) extends smoothly across the diagonal near
\((t_0,t_0)\).  The universal pole yields \(\Ob(u,v)>0\) off the diagonal in a sufficiently
small neighborhood of \((t_0,t_0)\), so the tangent web does not satisfy
the Samuelson condition there.
We call \(\Res(t,t)\) the Schwarzian residue and show that
\[
        \Res(t,t)
        =
        \frac12\Sch{\sigma}{t}
        =
        -\kaff(\sigma)
        \left(\frac{\dd\sigma}{\dd t}\right)^2.
\]
Here \(\sigma\) is the equi-affine arclength parameter of the envelope
\(\gamma(t)=\bigl(-h'(t),\,h(t)-t h'(t)\bigr)\), and
\(\Sch{\sigma}{t}\) denotes the Schwarzian derivative of \(\sigma\)
with respect to \(t\).
The equi-affine curvature \(\kaff\) is defined by
\mbox{\(\gamma_{\sigma\sigma\sigma}+\kaff\gamma_\sigma=0\)}.
We also derive the transformation law of the residue under
reparametrization of the tangent-line family.
\end{abstract}
\maketitle
\enlargethispage{2pt}

\section{Introduction}
\label{sec:intro}
The Samuelson condition is a classical area-ratio condition for planar
Lagrangian \(2\)-webs.  For the tangent Lagrangian \(2\)-webs considered here,
its failure follows from the planar nonflatness theorem of
Domitrz--Zubilewicz~\cite[Section~4.2]{DomitrzZubilewicz2024}.
In \cite{KurokawaUniversal}, we analyzed the limit in which two
tangents to a curve coalesce and recovered this failure
near the diagonal by showing that the Samuelson obstruction has
a universal double pole.

The present paper computes the diagonal finite part that remains after
subtracting this pole.  We express it as one half of the Schwarzian
derivative of equi-affine arclength with respect to the tangent slope,
and equivalently in terms of the equi-affine curvature of the envelope.
We also determine its transformation law under reparametrization of
the tangent-line family.

We first recall the Samuelson condition and its relation to the Hess
connection.  Let
\(\Omega\subset\mathbb R^2_{x,y}\) be a sufficiently small domain with
the standard symplectic, or area, form \(\omega=dx\wedge dy\), and let
\(\Web=(\mathcal F_1,\mathcal F_2)\) be a planar Lagrangian \(2\)-web
on \(\Omega\), that is, a pair of transverse Lagrangian foliations; see \cite{BlaschkeBol,Chern,Tabachnikov}.
Choose smooth first integrals \(u,v:\Omega\to\mathbb R\) with
\(du\wedge dv\neq0\).  After shrinking \(\Omega\), the map
\(\phi=(u,v):\Omega\to U\subset\mathbb R^2_{u,v}\) is a
diffeomorphism; write \(F=\phi^{-1}:U\to\Omega\).

Then
\(F^*\omega=J_F(u,v)\,du\wedge dv\),
where \(J_F\) denotes the Jacobian determinant of \(F\).  Since \(F\)
is a local diffeomorphism, \(J_F\) is nonzero, and \(|J_F|\) is the
unsigned area density in web coordinates.

The Samuelson condition requires that, for every sufficiently small
web rectangle and every subdivision by one leaf in each web direction,
the products of the areas in the two opposite pairs of subrectangles
be equal.  In smooth web coordinates, this is equivalent
to the local factorization \(|J_F(u,v)|=a(u)b(v)\) for some positive
one-variable functions \(a\) and \(b\).  Equivalently,
\[
        \partial_u\partial_v\log|J_F(u,v)|=0.
\]
For the equivalence with the area condition, see
\cite{CooperRussellSamuelson,FerraraUdriste}.  The use of \(|J_F|\)
makes the formulation independent of the orientation of the web
coordinates.

\medskip
\noindent\textit{Connection-theoretic interpretation.}
The two coordinate foliations on \(U\), together with \(F^*\omega\),
define a bi-Lagrangian structure.  Let \(\nabla^H\) denote its Hess
connection, also called the canonical bi-Lagrangian connection, namely
the unique torsion-free connection preserving the two coordinate
foliations and satisfying \(\nabla^H(F^*\omega)=0\)
\cite{Hess,Tabachnikov}.  Equivalently, the bi-Lagrangian structure is
para-K\"ahler, and \(\nabla^H\) is the Levi-Civita connection of its
associated neutral metric
\cite{EtayoSantamaria,EtayoSantamariaTrias}.

In the frame \((\partial_u,\partial_v)\), the mixed terms satisfy
\(\nabla^H_{\partial_u}\partial_v=
\nabla^H_{\partial_v}\partial_u=0\), while
\[
        \nabla^H_{\partial_u}\partial_u
        =
        \bigl(\partial_u\log|J_F|\bigr)\partial_u,
        \qquad
        \nabla^H_{\partial_v}\partial_v
        =
        \bigl(\partial_v\log|J_F|\bigr)\partial_v.
\]
Indeed, preservation of the two line fields and torsion-freeness force
the mixed terms to vanish, while \(\nabla^H(F^*\omega)=0\) determines
the remaining coefficients.

Set \(S_F(u,v):=\partial_u\partial_v\log|J_F(u,v)|\).  With the
curvature convention
\[R^H(X,Y)Z=\nabla^H_X\nabla^H_YZ-\nabla^H_Y\nabla^H_XZ-
\nabla^H_{[X,Y]}Z,\] a direct calculation gives
\[
        R^H(\partial_u,\partial_v)\partial_u
        =-S_F\partial_u,
        \qquad
        R^H(\partial_u,\partial_v)\partial_v
        =S_F\partial_v.
\]
Equivalently, relative to the same frame, the curvature \(2\)-form
matrix is
\[
        \mathcal R^H
        =
        S_F\,du\wedge dv
        \begin{pmatrix}
                -1&0\\
                0&1
        \end{pmatrix}.
\]
With the ordering of the two foliations fixed, the Samuelson
obstruction form \(S_F\,du\wedge dv\) is invariant under independent
reparametrizations of the two web coordinates.  In particular, the
Samuelson condition is equivalent to flatness of the Hess connection.
For this canonical-connection interpretation and its coordinate
formulas, see Tabachnikov~\cite{Tabachnikov} and
Domitrz--Zubilewicz~\cite[Section~2.3]{DomitrzZubilewicz2023}.

The same mixed logarithmic derivative occurs in the planar
nonuniformity tensor of Domitrz--Zubilewicz
\cite[Definition~2]{DomitrzZubilewicz2023}.
Their invariant is a symmetric covariant \(2\)-tensor, not the
alternating \(2\)-form used here; its planar coefficient nevertheless
expresses the same local obstruction.

\medskip

We now specialize to tangent webs.  Let
\(L_t:y=tx+h(t)\), \(t\in I\), be a smooth one-parameter family of
tangent lines.  For \(u\neq v\), let
\(\Psi(u,v)=(x(u,v),y(u,v))\) be the intersection point of \(L_u\)
and \(L_v\).  On a suitable local ordered branch, the parameters
\(u\) and \(v\) serve as web coordinates for the tangent \(2\)-web,
and \(\Psi\) is the corresponding inverse web-coordinate map.
Writing \(\Jac\) for the Jacobian determinant of \(\Psi\), set
\(\Ob(u,v) := \partial_u\partial_v\log|\Jac(u,v)|\).
By the preceding discussion, \(\Ob\) is the scalar coefficient in the
curvature \(2\)-form matrix of the Hess connection on that branch.

Near a diagonal point \((t_0,t_0)\) satisfying \(h''(t_0)\neq0\), the
smooth extension of \(\Jac\) vanishes to first order along the
diagonal.  Consequently,
\[
        \Ob(u,v)
        =
        \frac1{(v-u)^2}+O(1).
\]
Since \(\Psi\) ceases to be a local diffeomorphism on the diagonal,
the corresponding Hess connection is defined only off the diagonal.

The planar theorem of Domitrz--Zubilewicz
\cite[Section~4.2]{DomitrzZubilewicz2024} assumes transverse tangent
lines, a joining chord transverse to both curve germs, and nonzero
Euclidean curvature at the contact points.  For distinct \(u,v\) in a
sufficiently small interval about \(t_0\), the germs of the envelope at
these parameters satisfy the hypotheses, as verified in
Section~\ref{sec:intersection}.  The Jacobian and curvature formulas
used below are the corresponding slope-coordinate forms of their
tangent-intersection formulas.

The simple-zero factorization from \cite{KurokawaUniversal} is the
starting point for the diagonal calculation.  We recall its proof and
the smooth-extension argument for completeness.  We establish the
reparametrization law in Proposition~\ref{prop:change-parameter}.

The following two theorems state the finite-part formula and its
geometric interpretation.

\begin{theorem}
\label{thm:finitepart}
Let \(I\subset\mathbb R\) be an open interval, let
\(h\in C^\infty(I)\), and consider the family
\(L_t:y=tx+h(t)\).  Let \(\Psi(u,v)\) be the intersection map of
\(L_u\) and \(L_v\), and set
\(\Ob(u,v):=\partial_u\partial_v\log|\Jac(u,v)|\)
where \(\Jac\neq0\).  If \(t_0\in I\) and \(h''(t_0)\neq0\), then,
for \(u\neq v\) in a sufficiently small neighborhood of \((t_0,t_0)\),
\[
        \Ob(u,v)
        =
        \frac1{(v-u)^2}+\Res(u,v),
\]
where \(\Res\) extends smoothly to the diagonal.  Its diagonal value is
\[
        \Res(t,t)
        =
        \frac{h^{(4)}(t)}{3h''(t)}
        -
        \frac49
        \left(
                \frac{h'''(t)}{h''(t)}
        \right)^2.
\]
\end{theorem}

\begin{theorem}
\label{thm:schwarzian}
Under the hypotheses of Theorem~\ref{thm:finitepart}, let
\(\gamma(t)=\bigl(-h'(t),\,h(t)-t h'(t)\bigr)\) be the envelope.
On a sufficiently small interval about \(t_0\), let \(\sigma\) be its
equi-affine arclength parameter, so that
\(\dd\sigma=|h''(t)|^{2/3}\,\dd t\).  Then
\[
        \Res(t,t)
        =
        \frac12\Sch{\sigma}{t}.
\]
Equivalently,
\[
        \Res(t,t)
        =
        -\kaff(\sigma)
        \left(
                \frac{\dd\sigma}{\dd t}
        \right)^2,
\]
where \(\kaff\) is its equi-affine curvature, with the convention
\mbox{\(\gamma_{\sigma\sigma\sigma}+\kaff\gamma_\sigma=0\)}.
\end{theorem}

We call \(\Res(t,t)\) the \emph{Schwarzian residue} of the tangent web
at \(t\).  The relation between equi-affine curvature and the Schwarzian
derivative of the tangent slope is classical; we recall its proof in
Section~\ref{sec:affine}.  The finite part depends on the tangent
parameter because its definition subtracts the pole \(1/(v-u)^2\).
Section~\ref{sec:transformation} gives the resulting transformation law.

\section{Tangent intersection maps and the envelope}
\label{sec:intersection}

Let $I\subset\R$ be an open interval and let $h\in C^\infty(I)$.
We consider the one-parameter family of lines
\(L_t:\, y=tx+h(t)\).
For \(u\neq v\), the intersection point of \(L_u\) and \(L_v\) is
\[
        \Psi(u,v)=
        \left(
            \frac{h(v)-h(u)}{u-v},
            \frac{u h(v)-v h(u)}{u-v}
        \right).
\]
Where \(\Jac\neq0\), we restrict \(\Psi\) to a local ordered branch
on which it is a diffeomorphism onto its image.
On this branch, we write
\(\Psi(u,v)=(x(u,v),y(u,v))\).

The envelope of the family is
\[
        \env(t)=\bigl(-h'(t),\,h(t)-t h'(t)\bigr).
\]
Indeed, differentiating the equation \(y=tx+h(t)\) with respect to \(t\)
gives
\(0=x+h'(t),\)
hence \(x=-h'(t)\) and \(y=h(t)-t h'(t)\).

By the standard smooth extension of divided differences, the formula
above for \(\Psi\) extends smoothly to the diagonal.  Denoting the
extension by the same symbol, we have
\(\Psi(t,t)=\gamma(t)\).

Differentiating the envelope, we obtain
\(\env_t(t) = -h''(t)(1,t)\),
and
\(\env_{tt}(t) = \bigl(-h'''(t),-h''(t)-t h'''(t)\bigr)\).
Hence
\(\det(\env_t(t),\env_{tt}(t))=h''(t)^2\).
Thus the hypothesis \(h''(t)\neq0\) appearing in the main results is
precisely the equi-affine nondegeneracy condition for the envelope.

Define
\(A(u,v) := h(v)-h(u)-h'(u)(v-u)\),
and
\(B(u,v) := h(u)-h(v)+h'(v)(v-u)\).

The following Jacobian and curvature expressions are the
slope-coordinate forms of the tangent-intersection formulas in
\cite[Section~4.2]{DomitrzZubilewicz2024}.  The Jacobian factorization
is also used in \cite{KurokawaUniversal}.  We give a direct derivation
with our sign convention.

\begin{lemma}
\label{lem:jacobian}
For $u\neq v$, the Jacobian of the intersection map is
\[
        \Jac(u,v)
        =
        -\frac{A(u,v)B(u,v)}{(v-u)^3}.
\]
\end{lemma}

\begin{proof}
Write $d=v-u$.  From the preceding formula for $\Psi=(x,y)$ one obtains
\[
        x_u=-\frac{A}{d^2},\qquad
        x_v=-\frac{B}{d^2},\qquad
        y_u=-\frac{vA}{d^2},\qquad
        y_v=-\frac{uB}{d^2}.
\]
Hence
\[
        \Jac=x_u y_v-x_v y_u
        =
        \frac{uAB-vAB}{d^4}
        =
        -\frac{AB}{d^3},
\]
which is the claimed factorization.
\end{proof}

Taking \(\partial_u\partial_v\log|\Jac|\) gives, wherever \(AB\neq0\),
\[
        \Ob(u,v)
        =
        \frac{h''(u)B(u,v)}{A(u,v)^2}
        +
        \frac{h''(v)A(u,v)}{B(u,v)^2}
        -
        \frac{3}{(v-u)^2}.
\]
We use the factorized Jacobian below to separate the singular and
smooth parts of this curvature coefficient.

\medskip
\noindent\textit{Local two-tangent branches.}
Fix \(t_0\) with \(h''(t_0)\neq0\) and restrict to an open interval
\(I_0\) on which \(h''\) has constant nonzero sign.  Taylor's theorem
shows that \(A(u,v)\) and \(B(u,v)\) have that same sign for
\(u\neq v\) in \(I_0\), so Lemma~\ref{lem:jacobian} gives
\(\Jac(u,v)\neq0\).  For a fixed point \((x,y)\), the function
\(t\mapsto tx+h(t)-y\) is strictly convex or strictly concave on
\(I_0\); hence it has at most two zeros there.  For
\((x,y)=\Psi(u,v)\) with \(u<v\), the zeros are exactly \(u\) and \(v\).
Thus \((x,y)\) determines the ordered pair \((u,v)\) uniquely.
It follows that \(\Psi\) is injective and a local diffeomorphism on
\(\{(u,v)\in I_0^2:u<v\}\), and hence a diffeomorphism onto its image.
Here only tangents from the restricted family \(t\in I_0\) are counted.

For distinct \(u,v\in I_0\), put \(c=\gamma(u)-\gamma(v)\).
The identities
\[
\begin{aligned}
 \det(\gamma'(u),\gamma'(v))&=h''(u)h''(v)(v-u),\\
 \det(\gamma'(u),c)&=-h''(u)B(u,v),\qquad
 \det(\gamma'(v),c)=h''(v)A(u,v)
\end{aligned}
\]
show that the two tangents are transverse and the joining chord is
transverse to both curve germs.  The envelope has nonzero Euclidean
curvature at both contact points, since
\(\det(\gamma',\gamma'')=h''{}^2\neq0\).  The planar tangent-web theorem of
\cite[Section~4.2]{DomitrzZubilewicz2024} therefore applies: the Hess
curvature cannot vanish identically on a neighborhood of such a pair.

\section{The universal pole and its renormalization}
\label{sec:renormalization}

Where \(\Jac(u,v)\neq0\), set
\(\Ob(u,v) := \partial_u\partial_v\log |\Jac(u,v)|\).
The $2$-form $\Ob(u,v)\,du\wedge dv$ is the Samuelson obstruction form of the
tangent web; $\Ob$ is its coefficient in the coordinates $(u,v)$.

Let
\(
        d=v-u.
\)
In the auxiliary coordinates \((u,d)\), denote the partial derivatives by
\(\widehat{\partial}_u\) and \(\widehat{\partial}_d\).
We retain \(\partial_u\) and \(\partial_v\) for differentiation in the
original coordinates \((u,v)\).  Since \(d=v-u\), the chain rule gives
\[
        \partial_v=\widehat{\partial}_d,
        \qquad
        \partial_u=\widehat{\partial}_u-\widehat{\partial}_d.
\]
Write
\[
        \widetilde A(u,d):=A(u,u+d)
        =
        h(u+d)-h(u)-h'(u)d
\]
and
\[
        \widetilde B(u,d):=B(u,u+d)
        =
        h(u)-h(u+d)+h'(u+d)d.
\]

Differentiating with respect to $d$ in these auxiliary coordinates, we have
\[
        \widetilde A(u,0)=\widehat{\partial}_d\widetilde A(u,0)=0,
        \qquad
        \widehat{\partial}_d^2\widetilde A(u,0)=h''(u),
\]
and similarly
\[
        \widetilde B(u,0)=\widehat{\partial}_d\widetilde B(u,0)=0,
        \qquad
        \widehat{\partial}_d^2\widetilde B(u,0)=h''(u).
\]

By Hadamard's
lemma, there exist smooth functions $\alpha$ and $\beta$ such that
\[
        A(u,u+d)=d^2\alpha(u,d),
        \qquad
        B(u,u+d)=d^2\beta(u,d),
\]
with
\(\alpha(u,0)=\beta(u,0)=\frac12 h''(u).\)

For $d\neq0$, the Jacobian factorization gives
\[
        \Jac(u,u+d)
        =
        -\frac{A(u,u+d)B(u,u+d)}{d^3}
        =
        d\,Q(u,d),
\]
with
\(Q(u,d):=-\alpha(u,d)\beta(u,d).\)
Hence $Q$ is smooth and
\[
        Q(u,0)=-\frac14 h''(u)^2.
\]
In particular, if $h''(t_0)\neq0$, then $Q(t_0,0)\neq0$, and therefore
$Q$ is nonzero in a sufficiently small neighborhood of $(t_0,0)$.

This factorization allows us to separate the universal pole from a smooth remainder.

\begin{proposition}
Let \(t_0\in I\) satisfy \(h''(t_0)\neq0\).  Then the function
\(\partial_u\partial_v\log|\Jac(u,v)|-1/(v-u)^2\), defined off the
diagonal near \((t_0,t_0)\), admits a smooth extension to a
neighborhood of \((t_0,t_0)\).  We denote this extension by \(\Res\).
For \(u\neq v\) in this neighborhood,
\[
        \Ob(u,v)
        =
        \frac{1}{(v-u)^2}
        +
        \Res(u,v).
\]
\end{proposition}

\begin{proof}
In the coordinates \((u,d)\), where \(d=v-u\), we have
\(\Jac(u,u+d)=d\,Q(u,d)\), where
\(Q(u,0)=-\frac14h''(u)^2\).  Since \(Q\) is smooth and nonvanishing
near \((t_0,0)\), the function
\(G(u,d):=\log|Q(u,d)|\) is smooth there.  For \(d\neq0\), we have
\(\log|\Jac(u,u+d)|=\log|d|+G(u,d)\).

By the chain-rule identities displayed above,
\[
        \partial_u\partial_v\log|d|
        =
        (\widehat{\partial}_u-\widehat{\partial}_d)
        \widehat{\partial}_d\log|d|
        =
        (\widehat{\partial}_u-\widehat{\partial}_d)\frac1d
        =
        \frac1{d^2}.
\]
On the other hand, since \(G\) is smooth near \(d=0\), so is
\((\widehat{\partial}_u-\widehat{\partial}_d)
\widehat{\partial}_dG\).  Hence, in the coordinates \((u,d)\), the
required smooth extension is given by
\[
        \Res(u,u+d)
        =
        (\widehat{\partial}_u-\widehat{\partial}_d)
        \widehat{\partial}_dG(u,d).
\]
For \(d\neq0\), this gives
\(\Ob(u,u+d)=1/d^2+\Res(u,u+d)\).  Thus the stated decomposition holds
off the diagonal, while \(\Res\) is smooth across it.  This proves the
proposition and the smooth-extension assertion in
Theorem~\ref{thm:finitepart}.
\end{proof}

Since \(1+(v-u)^2\Res(u,v)\) is smooth and equals \(1\) on the
diagonal, it is positive in a sufficiently small neighborhood of
\((t_0,t_0)\).  Hence \(\Ob(u,v)>0\) off the diagonal in that
neighborhood, so the tangent web does not satisfy the Samuelson
condition there.

\medskip
\noindent\textit{Interpretation via the Hess connection.}
On the chosen ordered branch, the Hess curvature formula from
Section~\ref{sec:intro} becomes, in the frame \((\partial_u,\partial_v)\),
\[
        \mathcal R^H
        =
        \left(
                \frac1{(v-u)^2}+\Res(u,v)
        \right)
        du\wedge dv
        \begin{pmatrix}
                -1&0\\
                0&1
        \end{pmatrix}.
\]
The pole is therefore the singular part of the Hess curvature
coefficient.  We compute the diagonal value \(\Res(t,t)\) of the smooth
remainder in the next section.

\section{The diagonal finite part and the Schwarzian formula}
\label{sec:finitepart}
We now compute the diagonal value of the smooth remainder in
Theorem~\ref{thm:finitepart} and prove the Schwarzian identity in
Theorem~\ref{thm:schwarzian}.
For a local diffeomorphism $f=f(t)$, we use the convention
\[
        \Sch{f}{t}
        =
        \left(\frac{f''}{f'}\right)'
        -
        \frac12
        \left(\frac{f''}{f'}\right)^2=\frac{f'''}{f'}-\frac{3}{2}\left(\frac{f''}{f'}\right)^2
\]
for the Schwarzian derivative.  We shall use the chain rule
\[
        \Sch{f\circ g}{s}
        =
        \Sch{f}{t}\big|_{t=g(s)}\,g'(s)^2
        +
        \Sch{g}{s}.
\]
For background on the Schwarzian derivative and its role in projective
differential geometry, see
Ovsienko--Tabachnikov~\cite{OvsienkoTabachnikov}.

\begin{proof}[Proof of the diagonal formula and the Schwarzian identity]
Put
\(d=v-u.\)
From the preceding factorization we have
\(\Jac(u,u+d)=d\,Q(u,d)\),
where $Q$ is smooth and nonzero near $d=0$.  
We compute the diagonal value of the smooth remainder by expanding
\(G(u,d):=\log|Q(u,d)|\)
to second order in $d$.  All Taylor expansions below have smooth
remainders in \((u,d)\); in particular, a remainder denoted by
\(O(d^m)\) can here be written as \(d^m\) times a smooth function.
This justifies the subsequent differentiations at \(d=0\).

Recall that
\(A(u,u+d)=d^2\alpha(u,d)\), \(B(u,u+d)=d^2\beta(u,d)\),
and
\(Q(u,d)=-\alpha(u,d)\beta(u,d).\)
Using the definitions of $A$ and $B$ and expanding $h(u+d)$ and $h'(u+d)$ at
$u$, we first obtain
\[
        A(u,u+d)
        =
        \frac12h''(u)d^2
        +
        \frac16h'''(u)d^3
        +
        \frac1{24}h^{(4)}(u)d^4
        +
        O(d^5),
\]
and
\[
        B(u,u+d)
        =
        \frac12h''(u)d^2
        +
        \frac13h'''(u)d^3
        +
        \frac18h^{(4)}(u)d^4
        +
        O(d^5).
\]
Dividing these expansions by $d^2$ gives
\[
        \alpha(u,d)
        =
        \frac12 h''(u)
        +
        \frac16 h'''(u)d
        +
        \frac1{24}h^{(4)}(u)d^2
        +
        O(d^3),
\]
and
\[
        \beta(u,d)
        =
        \frac12 h''(u)
        +
        \frac13 h'''(u)d
        +
        \frac18h^{(4)}(u)d^2
        +
        O(d^3).
\]

Set
\(q(u):=\frac{h'''(u)}{h''(u)}.\)
Then
\[
        Q(u,d)
        =
        -\frac14 h''(u)^2
        \left[
        1+q(u)d+
        \left(
            \frac{h^{(4)}(u)}{3h''(u)}
            +
            \frac29q(u)^2
        \right)d^2
        +
        O(d^3)
        \right].
\]
After shrinking the neighborhood if necessary, $h''(u)\neq0$ and the bracket
is positive, since it is equal to $1$ at $d=0$.  Taking logarithms and using
\(\log(1+X)=X-\frac12X^2+O(X^3)\)
gives
\[
        G(u,d)=\log|Q(u,d)|
        =
        c(u)+q(u)d+r(u)d^2+O(d^3),
\]
where
\(c(u)=\log\left(\frac14 h''(u)^2\right)\)
and
\[
        r(u)
        =
        \left(
            \frac{h^{(4)}(u)}{3h''(u)}
            +
            \frac29q(u)^2
        \right)
        -
        \frac12q(u)^2
        =
        \frac{h^{(4)}(u)}{3h''(u)}
        -
        \frac5{18}q(u)^2.
\]

Since
\(\log|\Jac(u,u+d)|=\log|d|+G(u,d)\)
and the contribution of \(\log|d|\) to the mixed derivative is exactly
\(1/d^2\), it follows that, for \(d\neq0\), the smooth remainder is
given, in the original web coordinates, by
\[
        \Res(u,v)=\partial_u\partial_v\bigl[G(u,v-u)\bigr].
\]
Using the chain-rule identities from Section~\ref{sec:renormalization}
and then setting \(v=u+d\), this becomes
\[
        \Res(u,u+d)
        =
        (\widehat{\partial}_u-\widehat{\partial}_d)
        \widehat{\partial}_dG(u,d).
\]
Since $G$ is smooth up to $d=0$, the diagonal value is obtained by setting
$d=0$:
\[
        \Res(t,t)
        =
        \left.
        (\widehat{\partial}_u-\widehat{\partial}_d)
        \widehat{\partial}_dG(u,d)
        \right|_{u=t,\,d=0}.
\]
From
\(G(u,d)=c(u)+q(u)d+r(u)d^2+O(d^3)\),
we get
\(\widehat{\partial}_dG(u,d) = q(u)+2r(u)d+O(d^2)\).
Hence
\[
        \left.
        (\widehat{\partial}_u-\widehat{\partial}_d)
        \widehat{\partial}_dG(u,d)
        \right|_{d=0}
        =
        q'(u)-2r(u).
\]
Since
\(q'(t) = \frac{h^{(4)}(t)}{h''(t)} - q(t)^2,\)
we obtain
\[
\begin{aligned}
        \Res(t,t)
        &=
        q'(t)-2r(t)                                               \\
        &=
        q'(t)
        -
        \frac{2h^{(4)}(t)}{3h''(t)}
        +
        \frac59q(t)^2                                             \\
        &=
        \frac{h^{(4)}(t)}{3h''(t)}
        -
        \frac49q(t)^2                                             \\
        &=
        \frac13q'(t)-\frac19q(t)^2.
\end{aligned}
\]
The third line is the diagonal formula in
Theorem~\ref{thm:finitepart}, while the last line is the form needed
for the Schwarzian interpretation.

We continue to work near \(t_0\) under the hypothesis
\(h''(t_0)\neq0\).  After restricting to a smaller interval about
\(t_0\), we may assume that \(h''(t)\neq0\) throughout the interval.
By the computation in Section~\ref{sec:intersection},
\(\det(\env_t(t),\env_{tt}(t))=h''(t)^2\neq0\).
Thus the envelope is equi-affinely nondegenerate on this interval.
Let \(\sigma\) denote its equi-affine arclength parameter, defined by
\[
        \dd\sigma
        =
        |\det(\env_t(t),\env_{tt}(t))|^{1/3}\,\dd t.
\]
Then
\(\dd\sigma = |h''(t)|^{2/3}\,\dd t\).
Since \(h''\) is nowhere zero on the interval, it has fixed sign, and
therefore
\(\frac{\sigma''}{\sigma'} = \frac23\frac{h'''}{h''} = \frac23q\).

Therefore
\[
        \Sch{\sigma}{t}
        =
        \left(\frac{\sigma''}{\sigma'}\right)'
        -
        \frac12
        \left(\frac{\sigma''}{\sigma'}\right)^2
        =
        \frac23q'
        -
        \frac29q^2.
\]
Hence
\[
        \frac12\Sch{\sigma}{t}
        =
        \frac13q'
        -
        \frac19q^2
        =
        \Res(t,t).
\]
This completes the proof of
Theorem~\ref{thm:finitepart} and proves the Schwarzian identity in
Theorem~\ref{thm:schwarzian}.
\end{proof}

\section{Equi-affine curvature interpretation}
\label{sec:affine}

We now interpret the Schwarzian residue formula in the equi-affine
geometry of the envelope.  For the general background on equi-affine
differential geometry, see Nomizu--Sasaki~\cite{NomizuSasaki}; for the
plane-curve setting, see Izumiya--Sano~\cite{IzumiyaSano}.

For a plane curve \(\gamma(\sigma)\) parametrized by equi-affine arclength,
we choose the orientation so that
\(\det(\gamma_\sigma,\gamma_{\sigma\sigma})=1\).
With our sign convention, the equi-affine curvature \(\kaff\) is defined by
\(\gamma_{\sigma\sigma\sigma}+\kaff(\sigma)\gamma_\sigma=0\).
Equivalently, if \(T=\gamma_\sigma\) and \(N=\gamma_{\sigma\sigma}\), then
the equi-affine Frenet equations are \(T_\sigma=N\) and
\(N_\sigma=-\kaff(\sigma)T\).

For the envelope considered here, \(\gamma_t=-h''(t)(1,t)\).
On the interval under consideration, \(h''(t)\neq0\), or equivalently,
the envelope is equi-affinely nondegenerate.  Hence its tangent direction
is represented by \((1,t)\), so \(t\) is precisely the slope of the
tangent line to the envelope.  Let \(\sigma\) be the equi-affine
arclength parameter introduced in Section~\ref{sec:finitepart}.  Since
\(\frac{\dd\sigma}{\dd t}=|h''(t)|^{2/3}>0,\)
we may locally regard \(t\) as a function of \(\sigma\).

The following lemma relates equi-affine curvature to the Schwarzian
derivative of the tangent slope.

\begin{lemma}
\label{lem:affine-schwarzian}
Let $\gamma$ be parametrized by equi-affine arclength $\sigma$, with
\(\det(\gamma_\sigma,\gamma_{\sigma\sigma})=1\), and suppose that its
tangent line is nowhere vertical.  Let $t=t(\sigma)$ denote the slope
of this tangent line.  With the above curvature convention,
\[
        \kaff(\sigma)
        =
        \frac12\Sch{t}{\sigma}.
\]
Equivalently, for the inverse function $\sigma=\sigma(t)$,
\[
        \frac12\Sch{\sigma}{t}
        =
        -\kaff(\sigma)
        \left(\frac{\dd\sigma}{\dd t}\right)^2.
\]
\end{lemma}

\begin{proof}
Set \(T=\gamma_\sigma=(T_1,T_2)\).  By our curvature convention,
\(T_{\sigma\sigma}+\kaff(\sigma)T=0\), so both \(T_1\) and \(T_2\)
satisfy \(z_{\sigma\sigma}+\kaff(\sigma)z=0\).  The Wronskian of
\(T_1\) and \(T_2\) is
\[
        T_1(T_2)_\sigma-T_2(T_1)_\sigma
        =
        \det(T,T_\sigma)
        =
        \det(\gamma_\sigma,\gamma_{\sigma\sigma})
        =
        1,
\]
so they are linearly independent.

Since the tangent line is nowhere vertical, \(T_1\neq0\), and its slope is
\(t=T_2/T_1\).  Hence
\[
        t_\sigma
        =
        \frac{T_1(T_2)_\sigma-T_2(T_1)_\sigma}{T_1^2}
        =
        \frac{1}{T_1^2}
        >0,
\]
so \(t\) is a local parameter.  

The relation between ratios of solutions of second-order linear equations
and the Schwarzian derivative is classical; see, for example,
Hille~\cite[Theorem~10.1.1]{Hille}.  In the present setting, the required
identity follows directly from \(t_\sigma=T_1^{-2}\), which implies
\(t_{\sigma\sigma}/t_\sigma=-2(T_1)_\sigma/T_1\).  Hence
\[
\begin{aligned}
        \Sch{t}{\sigma}
        &=
        \left(
        -2\frac{(T_1)_\sigma}{T_1}
        \right)_\sigma
        -
        2
        \left(
        \frac{(T_1)_\sigma}{T_1}
        \right)^2                                      \\
        &=
        -2\frac{(T_1)_{\sigma\sigma}}{T_1}
        =
        2\kaff(\sigma),
\end{aligned}
\]
where the last equality uses
\((T_1)_{\sigma\sigma}+\kaff(\sigma)T_1=0\).
This proves the first formula.

For the inverse function \(\sigma=\sigma(t)\), the Schwarzian chain rule gives
\[
        0
        =
        \Sch{\sigma\circ t}{\sigma}
        =
        \Sch{\sigma}{t}
        \left(\frac{\dd t}{\dd\sigma}\right)^2
        +
        \Sch{t}{\sigma}.
\]
Combining this identity with the first formula proves the second.
\end{proof}

Combining the Schwarzian identity proved in Section~\ref{sec:finitepart} with the preceding
lemma gives the equi-affine curvature formula in Theorem~\ref{thm:schwarzian}:
\[
        \Res(t,t)
        =
        -\kaff(\sigma)
        \left(\frac{\dd\sigma}{\dd t}\right)^2.
\]
This completes the proof of
Theorem~\ref{thm:schwarzian}.

\section{Transformation law}
\label{sec:transformation}

The diagonal finite part depends on the parameter used to label the
tangent lines.  
Let \(\widetilde I\subset\mathbb R\) be an open interval, and let
\(\varphi:\widetilde I\longrightarrow I\)
be a smooth local diffeomorphism.  Set \(t=\varphi(s)\).  The same local
family of lines, now labeled by \(s\), is
\(\widetilde L_s:=L_{\varphi(s)}: \, y=\varphi(s)x+h(\varphi(s))\).
In the fixed affine coordinates \((x,y)\), the slope of
\(\widetilde L_s\) is
\(\varphi(s)\), not \(s\) in general.  Consequently, the slope-parameter
identity
\(\Res(t,t)=\frac12\Sch{\sigma}{t}\)
cannot in general be applied with \(s\) in place of \(t\).  We therefore
derive the transformation law directly from the reparametrized
intersection map.

For \(r\neq s\), the corresponding intersection map is
\(\widetilde\Psi(r,s) := \Psi\bigl(\varphi(r),\varphi(s)\bigr)\),
and we set
\(\widetilde{\Ob}(r,s) := \partial_r\partial_s \log\left|J_{\widetilde\Psi}(r,s)\right|\)
where \(J_{\widetilde\Psi}(r,s)\neq0\).

\begin{proposition}
\label{prop:change-parameter}
Let \(s_0\in\widetilde I\) satisfy
\(h''\bigl(\varphi(s_0)\bigr)\neq0.\)
Then, for \(r\neq s\) near \((s_0,s_0)\),
\[
        \widetilde{\Ob}(r,s)
        =
        \frac{1}{(s-r)^2}
        +
        \widetilde{\Res}(r,s),
\]
where \(\widetilde{\Res}\) extends smoothly across the diagonal near
\((s_0,s_0)\).
Its diagonal value satisfies
\[
        \widetilde{\Res}(s,s)
        =
        \Res\bigl(\varphi(s),\varphi(s)\bigr)\varphi'(s)^2
        +
        \frac16\Sch{\varphi}{s}.
\]
\end{proposition}

\begin{proof}
By the chain rule,
\[
        J_{\widetilde\Psi}(r,s)
        =
        \Jac\bigl(\varphi(r),\varphi(s)\bigr)
        \varphi'(r)\varphi'(s).
\]
Hence
\[
\begin{aligned}
        \log\left|J_{\widetilde\Psi}(r,s)\right|
        &=
        \log\left|
        \Jac\bigl(\varphi(r),\varphi(s)\bigr)
        \right|                               
        +
        \log|\varphi'(r)|
        +
        \log|\varphi'(s)|.
\end{aligned}
\]
The last two terms have zero mixed derivative, and therefore
\[
        \widetilde{\Ob}(r,s)
        =
        \varphi'(r)\varphi'(s)
        \Ob\bigl(\varphi(r),\varphi(s)\bigr).
\]
Using
\(\Ob(u,v) = \frac{1}{(v-u)^2} + \Res(u,v)\),
we obtain
\[
\begin{aligned}
        \widetilde{\Ob}(r,s)
        &=
        \frac{\varphi'(r)\varphi'(s)}
        {\bigl(\varphi(s)-\varphi(r)\bigr)^2}
        +
        \varphi'(r)\varphi'(s)
        \Res\bigl(\varphi(r),\varphi(s)\bigr).
\end{aligned}
\]
The first term is the coefficient of \(dr\wedge ds\) in the pullback of
\(\frac{1}{(v-u)^2}\,du\wedge dv\).
To define the finite part in the new parameters, we subtract
\(1/(s-r)^2\).  The resulting correction is
\[
        C_\varphi(r,s)
        :=
        \frac{\varphi'(r)\varphi'(s)}
        {\bigl(\varphi(s)-\varphi(r)\bigr)^2}
        -
        \frac{1}{(s-r)^2}.
\]
Then, for \(r\neq s\),
\[
\begin{aligned}
        \widetilde{\Ob}(r,s)
        &=
        \frac{1}{(s-r)^2}
        +
        C_\varphi(r,s)
        +
        \varphi'(r)\varphi'(s)
        \Res\bigl(\varphi(r),\varphi(s)\bigr).
\end{aligned}
\]

To prove that \(C_\varphi\) extends smoothly across the diagonal, write
\[
\begin{aligned}
        C_\varphi(r,s)
        &=
        \partial_r\partial_s
        \left(
        \log|\varphi(s)-\varphi(r)|
        -
        \log|s-r|
        \right)                                      \\
        &=
        \partial_r\partial_s
        \log\left|
        \frac{\varphi(s)-\varphi(r)}{s-r}
        \right|.
\end{aligned}
\]
The divided difference
\(\frac{\varphi(s)-\varphi(r)}{s-r}\)
extends smoothly to the diagonal with value \(\varphi'(r)\neq0\).
After shrinking the neighborhood if necessary, it remains nonzero.
Hence the logarithm of its absolute value extends smoothly across
the diagonal, and so does \(C_\varphi\).

It remains to compute the diagonal value of \(C_\varphi\).  Set
\(\delta=s-r, \;a=\frac{\varphi''(r)}{\varphi'(r)}, \; b=\frac{\varphi'''(r)}{\varphi'(r)}.\)
Expanding at \(r\), we obtain
\[
\begin{aligned}
        \frac{\varphi'(r)\varphi'(r+\delta)}
        {\bigl(\varphi(r+\delta)-\varphi(r)\bigr)^2}
        &=
        \frac{1}{\delta^2}
        \frac{
        1+a\delta+\frac12b\delta^2+O(\delta^3)
        }{
        \left(
        1+\frac12a\delta+\frac16b\delta^2+O(\delta^3)
        \right)^2
        }                                                   \\
        &=
        \frac{1}{\delta^2}
        +
        \frac16
        \left(
        b-\frac32a^2
        \right)
        +
        O(\delta)                                           \\
        &=
        \frac{1}{\delta^2}
        +
        \frac16\Sch{\varphi}{r}
        +
        O(\delta).
\end{aligned}
\]
Since \(C_\varphi(r,r+\delta)\) is obtained by subtracting
\(1/\delta^2\), it follows that
\[
        C_\varphi(r,r)=\frac16\Sch{\varphi}{r}.
\]

Returning to the decomposition above, for \(r\neq s\) we have
\[
        \widetilde{\Ob}(r,s)-\frac{1}{(s-r)^2}
        =
        C_\varphi(r,s)
        +
        \varphi'(r)\varphi'(s)
        \Res\bigl(\varphi(r),\varphi(s)\bigr).
\]
The first term on the right extends smoothly to the diagonal by the
preceding argument.  The second term also extends smoothly, since
\(\Res\) is smooth near
\(\bigl(\varphi(s_0),\varphi(s_0)\bigr)\).  Hence their sum gives the
smooth extension \(\widetilde{\Res}\).  Evaluating this extension on
the diagonal and using the formula for \(C_\varphi(r,r)\), we obtain
\[
        \widetilde{\Res}(r,r)
        =
        \Res\bigl(\varphi(r),\varphi(r)\bigr)\varphi'(r)^2
        +
        \frac16\Sch{\varphi}{r}.
\]
This is the stated transformation law.
\end{proof}

\begin{remark}
If
\(P(t):=6\Res(t,t), \; \widetilde P(s):=6\widetilde{\Res}(s,s),\)
then Proposition~\ref{prop:change-parameter} becomes
\[
        \widetilde P(s)
        =
        P\bigl(\varphi(s)\bigr)\varphi'(s)^2
        +
        \Sch{\varphi}{s}.
\]
Thus \(P\) has the transformation law of a projective connection
coefficient in this normalization; see \cite{OvsienkoTabachnikov}.
The Schwarzian residue is a parameter-dependent coefficient, whose
correction term reflects the subtraction of the standard pole in each
parameter.

Combining Proposition~\ref{prop:change-parameter} with
\(\Res(t,t)=\frac12\Sch{\sigma}{t}\)
and using the Schwarzian chain rule
\[
        \Sch{\sigma\circ\varphi}{s}
        =
        \left.
        \Sch{\sigma}{t}
        \right|_{t=\varphi(s)}
        \varphi'(s)^2
        +
        \Sch{\varphi}{s},
\]
we obtain
\[
\begin{aligned}
        \widetilde{\Res}(s,s)
        &=
        \frac12
        \left.
        \Sch{\sigma}{t}
        \right|_{t=\varphi(s)}
        \varphi'(s)^2
        +
        \frac16\Sch{\varphi}{s}                         \\
        &=
        \frac12\Sch{\sigma\circ\varphi}{s}
        -
        \frac13\Sch{\varphi}{s}.
\end{aligned}
\]
In particular, if \(\varphi\) is a M\"obius reparametrization, then
\(\Sch{\varphi}{s}=0\), and hence
\(\widetilde{\Res}(s,s) = \frac12\Sch{\sigma\circ\varphi}{s}\).
For the same M\"obius reparametrization,
\[
        \widetilde{\Res}(s,s)\dd s ^2
        =
        \Res(t,t)\dd t ^2,
        \qquad
        t=\varphi(s).
\]
\end{remark}

\section{Examples}
\label{sec:examples}

\begin{example}[Parabola]
Let \(h(t)=-\frac12t^2\).  The envelope is
\(\gamma(t)=\left(t,\frac12t^2\right),
\)
which parametrizes the parabola \(y=\frac12x^2\).
Then
\(A(u,v)=-\frac12(v-u)^2, \; B(u,v)=-\frac12(v-u)^2.\)
Hence
\(\Jac(u,v) = -\frac{A(u,v)B(u,v)}{(v-u)^3} = -\frac{v-u}{4}\).
Therefore
\[
        \log|\Jac(u,v)|=\log|v-u|-\log 4,
        \qquad
        \Ob(u,v)=\frac1{(v-u)^2}.
\]
Thus
\(\Res(u,v)\equiv0.\)
The parabola has zero equi-affine curvature.  The Hess connection is
nonflat, since \(\Ob(u,v)=1/(v-u)^2\neq0\).
\end{example}

\begin{example}[Circle]
Take \(h(t)=\sqrt{1+t^2}\).  Its envelope is
\(\gamma(t)=(-t,1)/\sqrt{1+t^2}\), which parametrizes the upper
semicircle \(x^2+y^2=1\), \(y>0\).
Then
\(h''(t)=(1+t^2)^{-3/2}, \;\frac{d\sigma}{dt}=|h''(t)|^{2/3}=(1+t^2)^{-1}.\)
Thus
\(\sigma(t)=\arctan t\)
up to an additive constant.  The unit circle has equi-affine curvature
\(\kappa_{\mathrm{aff}}=1\)
with our sign convention.  Hence the equi-affine curvature formula for the residue gives
\[
        \Res(t,t)
        =
        -\kappa_{\mathrm{aff}}
        \left(\frac{d\sigma}{dt}\right)^2
        =
        -\frac{1}{(1+t^2)^2}.
\]
Equivalently,
\(\Res(t,t)=\frac12\Sch{\arctan t}{t}.\)
The residue varies with \(t\) through the factor \((d\sigma/dt)^2\),
even though the equi-affine curvature is constant.
\end{example}

\begin{example}[A quartic perturbation of the parabola]
To see how the residue changes under a quartic perturbation of the
parabola, consider
\(h(t)=-\frac12t^2+\frac{\varepsilon}{24}t^4\).
At \(t=0\), we have \(h''(0)=-1\), \(h'''(0)=0\), and
\(h^{(4)}(0)=\varepsilon\).  Hence the diagonal formula gives
\(\Res(0,0)=-\frac{\varepsilon}{3}\).
Thus the residue at \(t=0\) depends linearly on the perturbation
parameter \(\varepsilon\).

\end{example}

\begin{example}[A cubic model with a cuspidal envelope]
Let \(h(t)=\frac16t^3\).  Its envelope is
\(\gamma(t)=\left(-\frac12t^2,-\frac13t^3\right)\), which has a cusp
at \(t=0\).  On any interval not containing \(0\), \(h''(t)=t\) does
not vanish, so the envelope is equi-affinely nondegenerate.  The
diagonal formula gives
\(\Res(t,t)=-\frac4{9t^2}\).  For \(t>0\), we have
\(\dd\sigma=t^{2/3}\dd t\) and
\(\sigma(t)=\frac35t^{5/3}\), and hence
\(\frac12\Sch{\sigma}{t}=-\frac4{9t^2}\).
Thus the Schwarzian residue diverges to \(-\infty\) as the contact
point approaches the cusp at \(t=0\).
\end{example}


\end{document}